\documentclass[dvips]{article}
\usepackage[latin1]{inputenc}
\usepackage{amsmath, amssymb, amsthm}
\usepackage{color}
\usepackage[dvips]{graphicx}
\usepackage{mathrsfs}
\usepackage[usenames,dvipsnames,svgnames,table]{xcolor}

\numberwithin{equation}{section}

\newcommand{\ind}[1]{\mathbf{1}_{#1}}

\newtheorem{theo}{Theorem}[section]
\newtheorem{pr}{Proposition}[section]

\newtheorem{lem}{Lemma}[section]

\newcommand{\Z}{\mathbb{Z}}
\newcommand{\N}{\mathbb{N}}
\newcommand{\R}{\mathbb{R}}

\begin{document}

\title{A Quenched Functional Central Limit Theorem for Planar Random Walks in Random Sceneries}
\author{Nadine Guillotin-Plantard \footnotemark[1] \ , Julien Poisat \footnotemark[2]
 \\ and Renato Soares dos Santos \footnotemark[1] }

\footnotetext[1]{Institut Camille Jordan, CNRS UMR 5208, Universit\'e de Lyon, Universit\'e Lyon 1, 43, Boulevard du 11 novembre 1918, 69622 Villeurbanne, France.\\ E-mail: nadine.guillotin@univ-lyon1.fr ; soares@math.univ-lyon1.fr;}
\footnotetext[2]{Mathematical Institute, Leiden University, P.O. Box 9512, 2300 RA Leiden, The Netherlands.
E-mail: poisatj@math.leidenuniv.nl; 
This work was supported by the french ANR project MEMEMO2 10--BLAN--0125--03 and ERC Advanced Grant 267356 VARIS.\\
}

\maketitle

\begin{abstract}
Random walks in random sceneries (RWRS) are simple examples of stochastic processes in disordered media. They were introduced at the end of the 70's by Kesten-Spitzer and Borodin, motivated by the construction of new self-similar processes with stationary increments. Two sources of randomness enter in their definition: a random field $\xi = (\xi(x))_{x \in \Z^d}$ of i.i.d.\ random variables, which is called the \emph{random scenery}, and a random walk $S = (S_n)_{n \in \N}$ evolving in $\Z^d$, independent of the scenery. The RWRS $Z = (Z_n)_{n \in \N}$ is then defined as the accumulated scenery along the trajectory of the random walk, i.e., $Z_n := \sum_{k=1}^n \xi(S_k)$. The law of $Z$ under the joint law of $\xi$ and $S$ is called ``annealed'', and the conditional law given $\xi$ is called ``quenched''. Recently, functional central limit theorems under the quenched law were proved for $Z$ by the first two authors for a class of transient random walks including walks with finite variance in dimension $d \ge 3$. In this paper we extend their results to dimension $d=2$.

\end{abstract}

\section{Introduction}
Let $d\geq1$ and $(\xi(x))_{x\in \Z^d}$ be a collection of independent and identically distributed (i.i.d.)\ real random variables,
further referred to as {\it scenery}, and $(S_n)_{n\geq0}$ a random walk evolving in $\Z^d$, independent of the scenery. The random walk in random scenery (RWRS) is the process obtained by adding up the values of the scenery seen by the random walk along its trajectory, that is, $Z_n = \xi(S_1) + \ldots + \xi(S_n)$, $n\geq1$. This model was introduced independently by Kesten and Spitzer \cite{KS79} and by Borodin \cite{Bor79-1,Bor79-2}.\\

RWRS appears naturally in a variety of contexts, for instance (i) in the energy function of statistical mechanics models of polymers interacting with a random medium, (ii) in Bouchaud's trap model via the clock process, see \cite{BAC07}, (iii) in the study of random walks in randomly oriented lattices, as in \cite{CP03, CGPPS11}. The last example is related to the phenomenon of anomalous diffusion in layered random media, see Le Doussal \cite{lD92} and Matheron and de Marsily \cite{MdM80} on this matter. Indeed, Kesten and Spitzer's original motivation was to build a new class of self-similar sochastic processes with non-standard normalizations.\\

Results were first established under the {\it annealed} measure, that is when one averages at the same time over the scenery and the random walk.
Let us suppose here that the random walk increment and the scenery at the origin are in the domains of attraction of different stable laws with index $a$ and $b$ $\in (0, 2]$, respectively. In the case $d = 1 < a$, Kesten and Spitzer \cite{KS79} proved that the process $(n^{-\delta}Z_{\lfloor nt \rfloor})_{t\geq0}$ converges weakly, as $n\to\infty$, to a continuous $\delta$-self-similar process, where $\delta =  1 - a^{-1} + (ab)^{-1}$. Later on, Bolthausen \cite{Bo89} proved a functional central limit for $(\sqrt{n\log n}^{\hspace*{1pt}-1} Z_{\lfloor nt \rfloor})_{t\geq0}$ in the case $d=a=b=2$, 
and his result also covers the case $d=a=1$, $b=2$. More recently, Castell, Guillotin-Plantard and P\`ene \cite{CGPP} proved that, for $d=a\in\{1,2\}$ and $0<b<2$, $(Z_{\lfloor nt \rfloor})_{t\geq0}$ has to be normalized by $n^{1/b}(\log n)^{1-1/b}$ so that it converges to a limiting process, which is stable of index $b$. The case of a transient random walk (i.e., $a < d$) has also been treated in \cite{CGPP} (see also \cite{S76,KS79,Bor79-2}): rescaling by $n^{1/b}$ one obtains as limit a stable process of index $b$. Other results on RWRS include strong approximation results and laws of the iterated logarithm \cite{CKS99, CR83, KL98}, limit theorems for correlated sceneries or walks \cite{CD09,GPP10}, large and moderate deviations results \cite{AC07, C04, CP01, FlMoWa08, GKS07}, ergodic and mixing properties \cite{dHS06}.\\

Distributional limit theorems for {\it quenched} sceneries (that is, conditionally given the scenery) are more recent. The first result in this direction that we are aware of was obtained by Ben Arous and \v{C}ern\'y \cite{BAC07}, in the case of a heavy-tailed scenery and planar random walk. Recently, the first two authors proved in \cite{GuPo12} that a quenched functional central limit theorem (with the usual $\sqrt{n}$-scaling and Gaussian law in the limit) holds for a class of transient random walks. Moreover, with one of the methods used there, namely convergence of moments, they could prove convergence along a subsequence for sceneries having finite moments of all orders and planar random walks with finite non-singular covariance matrices, after a non-standard scaling by $\sqrt{n\log n}$. The question was raised whether the convergence takes place along the full sequence. In this paper we are able to answer this question in the positive when the scenery has slightly more than a second moment.\\

\section{Notation, assumptions and results}
\label{sec:not}

Let us start with a few words about notation. 
We will denote by $\N := \{0,1,2,\ldots\}$
the set of non-negative integers and put $\N^* := \N \setminus \{0\}$.
We will write $C$ to denote a generic positive constant
that may change from expression to expression.

We now proceed to define the model.
Let $S = (S_n)_{n \ge 0}$ be a random walk in $\Z^2$ starting at $0$,
i.e., $S_0 = 0$ and
\begin{equation}
\left(S_n-S_{n-1}\right)_{n \ge 1} \text{ is a sequence of i.i.d.\ } \Z^2 \text{-valued random variables}.
\end{equation}
We denote the local times of the random walk by
\begin{equation}\label{defN}
N_n(x) := \sum_{1\leq k \leq n} \ind{\{S_k = x\}},\quad x\in\mathbb{Z}^2.
\end{equation}

Let $\xi = (\xi(x))_{x \in \Z^2}$ be a field of i.i.d.\ real random variables independent of $S$.
The field $\xi$ is called the \emph{random scenery}.

The random walk in random scenery (RWRS) $Z = (Z_n)_{n \ge 0}$ is defined 
by setting $Z_0 := 0$ and, for $n \in \N^{*}$,
\begin{equation}\label{defZ}
Z_n := \sum_{i=1}^n \xi(S_i) = \sum_{x \in \Z^2} \xi(x) N_n(x).
\end{equation}

We will denote by $\mathbb{P}$ the joint law of $S$ and $\xi$, and by $P$ the marginal of $S$.
The law $\mathbb{P}$ is called the \emph{annealed} law, while the conditional law $\mathbb{P}(\cdot | \xi)$ is
called the \emph{quenched} law.

We will make the following two assumptions on the random walk and on the random scenery:

\paragraph{(A1)}
The random walk increment $S_1$ has a centered law with a finite and non-singular covariance matrix $\Sigma$.
We further suppose that the random walk is aperiodic in the sense of Spitzer \cite{S76},
which means that $S$ is not confined to a proper subgroup of $\mathbb{Z}^2$.

\paragraph{(A2)} $\mathbb{E}[\xi(0)]=0$, $\mathbb{E}[|\xi(0)|^2]=1$ and there exists a ${ \chi > 0}$ such that
\begin{equation}\label{eq:A2}
\mathbb{E}\left[ |\xi(0)|^2 (\log^+ |\xi(0)|)^{ \chi}\right] < \infty,
\end{equation}
where $\log^+ x := \max(0,\log x)$.

\vspace{0.2cm}
The aim of this paper is to prove the following quenched functional central limit theorem. 
\begin{theo}\label{theoCLT}
Under assumptions (A1) and (A2), for $\mathbb{P}$-a.e.\ $\xi$, the process
\begin{equation}
W^{(n)}= \left( W_t^{(n)} \right)_{t \ge 0} := \left( \frac{Z_{\lfloor nt \rfloor} }{ \sqrt{n \log n}} \right)_{t \ge 0}
\end{equation}
converges weakly as $n \to \infty$ under $\mathbb{P}(\cdot | \xi)$ in the Skorohod topology to a Brownian motion
with variance $\sigma^2 = (\pi \sqrt{\det \Sigma})^{-1}$.
\end{theo}

\noindent\textbf{Remark:}
The conclusion of this theorem still holds if, alternatively, the assumption (A1) is replaced 
by the following one:
\paragraph{(A1')}   The sequence $S = (S_n)_{n \ge 0}$ is an aperiodic random walk in $\Z$ starting from $0$ 
such that $\left(\frac{S_n}n\right)_n$ converges
in distribution to a random variable with characteristic function given by 
$t\mapsto \exp(-a |t|)$, $a>0$. In this case, $\sigma^2 = 2 (\pi a)^{-1}$.

\vspace{0.2cm}

Indeed, the proof of Theorem \ref{theoCLT} 
depends on the random walk $S$ only through certain local time properties
which are known to be the same under assumptions (A1) or (A1'). 
These properties are listed in Section 4.

\section{Outline of the proof of Theorem~\ref{theoCLT}}
\label{sec:outlineproof}

We will use a method introduced by Bolthausen and Sznitman in \cite{BoSz02}. 
The idea is to pass the functional CLT from the annealed to the quenched law
using concentration of quenched expectations of Lipschitz functionals
of the rescaled process. In our setting, the annealed version of
Theorem~\ref{theoCLT} was proved by Bolthausen in \cite{Bo89}, and his proof also works
under (A1').

To describe the method more precisely, 
let $\mathcal{W}^{(n)}$ be the polygonal interpolation of $W^{(n)}$, that is,
\begin{equation}\label{def:interpol}
\mathcal{W}_t^{(n)} := \frac{Z_{\lfloor n t \rfloor} + \left(nt - \lfloor nt\rfloor \right) \left(Z_{\lfloor n t \rfloor + 1}-Z_{\lfloor n t \rfloor} \right)}{\sqrt{n \log n}}.
\end{equation}
For $T>0$, consider the space $C([0,T], \R)$ of continuous functions from $[0,T]$ to $\R$ equipped with the sup norm.
We abuse notation by writing $\mathcal{W}^{(n)}$ to mean also the restriction of this process to the interval $[0,T]$, depending on context.

Following the reasoning in Lemma 4.1 of \cite{BoSz02}, 
we see that Theorem~\ref{theoCLT} will follow from the annealed functional CLT in \cite{Bo89}
if we show that, for any $T>0$, $b \in (1,2]$ and any bounded Lipschitz function
$F:C([0,T], \R) \to \mathbb{C}$,
\begin{equation}\label{concentrationLipschitz}
\lim_{n \to \infty} \mathbb{E}\left[ F\left(\mathcal{W}^{(\lfloor b^n\rfloor)}\right) \;\middle|\; \xi \; \right] - 
\mathbb{E}\left[ F\left(\mathcal{W}^{(\lfloor b^n\rfloor)} \right) \right] = 0 \;\;\; \mathbb{P} \text{-a.s.}
\end{equation}

To prove \eqref{concentrationLipschitz} we will use 
a martingale decomposition, in a similar fashion as in Bolthausen and Sznitman \cite{BoSz02} (proof of Theorem 4.2), 
Berger and Zeitouni \cite{BeZe08} (proof of Theorem 4.1) and Rassoul-Agha and Sepp\"al\"ainen \cite{RaSe09} (proof of Proposition 6.1).
In order to control the martingale via exponential inequalities, 
we introduce first as a technical step a truncation of the scenery, from
which the restriction $\chi > 0$ originates.

The rest of the paper is organized as follows. 
In Section~\ref{sec:RW2d} we collect two facts about 
two-dimensional random walks that we will need. 
Section~\ref{sec:prooftheoCLT} contains the proof of Theorem~\ref{theoCLT}, 
given in two steps: in Section~\ref{subsec:trunc} we define a truncation
of the RWRS and reduce the problem to showing \eqref{concentrationLipschitz} for the truncated version,
and this last step is carried out in Section~\ref{subsec:control_truncation}.

\section{Two-dimensional random walks}
\label{sec:RW2d}

We state here two lemmas about two-dimensional random walks satisfying (A1) that will be needed in the sequel. 
Analogous statements are valid under (A1').

\begin{lem}\label{lemma:2dRWs}
There exists a $K \in (0,\infty)$ such that
\begin{align}
(i) & \quad \sup_{x \in \Z^2} E \left[ N_n(x)\right] \le K \log n \;\;\; \forall \; n \ge 2. \label{avloctimes} \\
(ii) & \quad \sum_{x \in \Z^2} E \left[ N_n(x) \right]^2 \le K n \qquad \forall \; n \in \N^*. \label{eq:mutintloctimes}
\end{align}
\end{lem}
\begin{proof}
Item \emph{(i)} can be found e.g.\ in the proof of Lemma 2.5 in \cite{Bo89}.
Item \emph{(ii)} follows from the proof of Corollary 3.2 in \cite{JFGR}; 
note that the l.h.s.\ of \eqref{eq:mutintloctimes} is the expectation of the mutual intersection
local time of two independent copies of $S$, denoted by $J_n$ in \cite{JFGR}.
\end{proof}

\begin{lem}\label{lem:exit_range}
Let
\begin{equation}\label{defRn}
R_n := \{x \in \Z^2 \colon \; N_n(x) > 0\}
\end{equation}
be the range of the random walk $S$ up to time $n$.
There exists a constant $C>0$ such that for all $n\geq 2$,
\begin{equation}
P(S_n \notin R_{n-1}) \leq C (\log n)^{-1}.
\end{equation}
\end{lem}
\begin{proof}
One can for instance find a proof in Section 2 of \cite{DE50}, which actually holds for more general random walks than the nearest-neighbour walk considered
there.
\end{proof}

\section{Proof of Theorem~\ref{theoCLT}}
\label{sec:prooftheoCLT}

The proof consists of two steps: 
first we define a truncation of the RWRS that approximates well the original process,
and then we prove \eqref{concentrationLipschitz} for the truncated version.

\subsection{Truncation}
\label{subsec:trunc}

For $n \ge 2$, set $M_n := \sqrt{n / (\log n)^\gamma}$, where
\begin{equation}
\gamma := 1 + \frac{\chi}{2}, 
\end{equation}

define $\xi_n$, $\widehat{\xi}_n$ $\in \R^{\Z^2}$ by
\begin{equation}\label{deftruncscene}
\begin{array}{rcl}
\xi_n(x) & := & \xi(x) \mathbf{1}_{\{|\xi(x)| \le M_n \}} \\
\widehat{\xi}_n(x) & := & \xi_n(x) - \mathbb{E} \left[ \xi_n(x) \right]
\end{array}
\;\;\; \text{ for }x \in \Z^2,
\end{equation}
and let $Z^{(n)}$ and $\widehat{Z}^{(n)}$ be defined by
\begin{equation}\label{deftruncpapa}
\begin{array}{rclcl}
Z^{(n)}_k & := & \sum_{i=1}^k \xi_n(S_i) & = & \sum_{x \in \Z^2} \xi_n(x) N_k(x) \\
\widehat{Z}^{(n)}_k & := & \sum_{i=1}^k \widehat{\xi}_n(S_i) & = & \sum_{x \in \Z^2} \widehat{\xi}_n(x) N_k(x)
\end{array}
\;\;\; \text{ for } k \in \N^*.
\end{equation}

The following two propositions show that, in order to prove Theorem~\ref{theoCLT},
it is enough to prove the same statement for $\widehat{W}^{(n)}_t:=  (n \log n)^{-\frac12} \widehat{Z}^{(n)}_{\lfloor n t\rfloor}$, $t \ge 0$.

\begin{pr}\emph{(Comparison between $Z$ and $Z^{(n)}$)}\label{prop:TvsNT}
\text{}

Fix $T>0$. There exists $\mathbb{P}$-a.s.\
a random time $T_0 \in \N^*$ such that, if $n \ge T_0$, 
then $Z_k^{(n)}= Z_k$ for all $1 \le k \le \lfloor n T \rfloor$.
\end{pr}
\begin{proof}
Let $R_k$ be the range of the random walk as in \eqref{defRn},
an set
\begin{equation}
\mathcal{D}_n := \{ x \in R_{\lfloor nT \rfloor} \colon\, \xi_n(x) \neq \xi(x)\}.
\end{equation}
We have
\begin{equation}
\mathcal{D}_n \setminus \mathcal{D}_{n-1} = 
\left\{x \in R_{\lfloor nT \rfloor} \setminus R_{\lfloor (n-1)T \rfloor} \colon\, |\xi(x)| > M_n \right\}.
\end{equation}
Therefore, if $d_n := \mathbb{P} \left( \mathcal{D}_n \setminus \mathcal{D}_{n-1} \neq \emptyset \right)$,
\begin{align}
d_n & = \mathbb{P} \left( \exists \; \lfloor (n-1) T \rfloor < k \le \lfloor n T \rfloor \colon\, |\xi(S_k)| > M_n,\, S_k \notin R_{\lfloor (n-1)T \rfloor}
\right) \nonumber\\
& \leq \mathbb{P}\Big(
\begin{array}{l} 
\exists \; \lfloor (n-1) T \rfloor < \ell \le \lfloor n T \rfloor \colon\, |\xi(S_\ell)| > M_n  \\
\mbox{and } \exists \; \lfloor (n-1) T \rfloor < k \le \lfloor n T \rfloor \colon\, S_k \notin R_{\lfloor (n-1)T \rfloor}
\end{array} \Big) \nonumber\\
& \leq \sum_{\ell = \lfloor (n-1) T \rfloor +1}^{\lfloor n T \rfloor}
\mathbb{P}\Big( \begin{array}{l} 
|\xi(S_\ell)| > M_n \mbox{ and }  \\
 \exists \; \lfloor (n-1) T \rfloor < k \le \lfloor n T \rfloor \colon\, S_k \notin R_{\lfloor (n-1)T \rfloor} \;
\end{array} \Big) \nonumber\\
& \le (T+1)\, \mathbb{P} \left( |\xi(0)| > M_n \right) \, \nonumber\\
& \qquad \qquad \;\; \times P\left( \exists \; \lfloor (n-1) T \rfloor < k \le \lfloor n T \rfloor \colon\, S_k \notin R_{\lfloor (n-1)T \rfloor} \right),
\end{align}
where the last inequality is justified by summing over the possible values of $S_{\ell}$.

Let us now prove that $(d_n)_{n\geq 1}$ is summable. 
Considering the first 
$k > \lfloor (n-1) T \rfloor$ such that $S_k \notin R_{\lfloor (n-1)T \rfloor}$,
we see that
\begin{align}\label{eq_ineq1_prop:TvsNT}
 & \;\; P\left( \exists \; \lfloor (n-1) T \rfloor < k \le \lfloor n T \rfloor \colon\, S_k \notin
R_{\lfloor (n-1)T \rfloor} \right) \nonumber\\
= & \;\; P\left( \exists \; \lfloor (n-1) T \rfloor < k \le \lfloor n T \rfloor \colon\, S_k \notin
R_{k-1} \right) \nonumber\\
\leq & \; \sum_{k = \lfloor (n-1) T \rfloor +1}^{\lfloor n T \rfloor} P(S_k \notin R_{k-1}) \leq \frac{C}{\log n},
\end{align}
where we used Lemma \ref{lem:exit_range} for the last inequality. On the other hand, since $f(x) := x^2 (\log^+ x )^{\chi}$ is non-decreasing on
$(0,\infty)$ and $f(M_n) \ge C n (\log n)^{\chi - \gamma}$ for some $C>0$ and all $n \ge 2$,
\begin{equation}\label{eq_ineq2_prop:TvsNT}
\mathbb{P}\left( |\xi(0)| > M_n \right) \le \mathbb{P}\left( |\xi(0)|^2 (\log^+|\xi(0)|)^\chi \ge C n (\log n)^{\chi - \gamma} \right).
\end{equation}
The combination of \eqref{eq_ineq1_prop:TvsNT} and \eqref{eq_ineq2_prop:TvsNT} yields
\begin{equation}
 \sum_{n\geq 2} d_n \leq C \sum_{n\geq 2} (\log n)^{-1} \mathbb{P}\left( |\xi(0)|^2 (\log^+|\xi(0)|)^\chi \ge C n (\log n)^{\chi - \gamma} \right).
\end{equation}
If $\chi - \gamma = \chi/2 - 1 \geq 0$, it follows from (A2) and the line above (where $(\log n)^{-1}$ is roughly bounded by a constant) that $(d_n)_{n\geq 1}$
is summable. We now restrict to the case $0<\chi<2$, that is $\chi - \gamma <0$, which requires a bit more work. For all $\alpha>0$,
\begin{align}
 \sum_{n \geq 3} d_n & \leq C \sum_{L\geq 1} \sum_{e^{L^{\alpha}}\leq n < e^{(L+1)^{\alpha}}} (\log n)^{-1} \mathbb{P}\left( |\xi(0)|^2
(\log^+|\xi(0)|)^\chi \ge C n (\log n)^{\chi - \gamma} \right)\nonumber\\
&\leq C \sum_{L\geq 1} L^{-\alpha} \sum_{e^{L^{\alpha}}\leq n < e^{(L+1)^{\alpha}}} \mathbb{P}\left( |\xi(0)|^2
(\log^+|\xi(0)|)^\chi \ge C L^{\alpha (\chi - \gamma)} n \right)\nonumber\\
&\leq C \sum_{L\geq 1} L^{-\alpha} \sum_{n\geq 1} \mathbb{P}\left( |\xi(0)|^2
(\log^+|\xi(0)|)^\chi \ge C L^{\alpha (\chi - \gamma)} n \right)\nonumber\\
& \leq C \sum_{L\geq 1} L^{-\alpha} \frac{1}{CL^{\alpha(\chi - \gamma)}} \mathbb{E} \left[|\xi(0)|^2 (\log^+|\xi(0)|)^\chi \right], 
\end{align}
which is finite as soon as $\alpha(1+\chi-\gamma)> 1$. Since $1+\chi - \gamma = \chi/2 >0$, the latter condition can be achieved by choosing $\alpha$ large
enough. We have now proven that $(d_n)_{n\geq 1}$ is summable, so by the Borel-Cantelli lemma there exists a random index $N_0 \in \N^*$
such that a.s.\ $\mathcal{D}_n \subset \mathcal{D}_{N_0}$ for all $n \ge N_0$.
Therefore, setting
\begin{equation}
T_0 := \inf \left\{n \ge N_0 \colon\, M_n > \sup_{x \in \mathcal{D}_{N_0}} |\xi(x)| \right\}, 
\end{equation}
we have $\mathcal{D}_n = \emptyset$ for $n \ge T_0$.
\end{proof}

\begin{pr}\label{prop:TvsTR}\emph{(Comparison between $Z^{(n)}$ and $\widehat{Z}^{(n)}$)}
\begin{equation}
\lim_{n \to \infty} \sup_{1\le k \le \lfloor nT \rfloor} \frac{|\widehat{Z}^{(n)}_k - Z^{(n)}_k|}{\sqrt{n \log n}} = 0 \;\;\; \mathbb{P} \text{-a.s.\ for any } T>0.
\end{equation}
\end{pr}
\begin{proof}
Since $\xi$ is centered,
\begin{align}\label{eq:est_mom_xin}
\left| \mathbb{E} \left[ \xi(0) \mathbf{1}_{\{|\xi(0)| \le M_n\}}\right] \right| = \left| \mathbb{E} \left[ \xi(0) \mathbf{1}_{\{|\xi(0)| > M_n\}}\right]
\right|
& \leq \frac{\mathbb{E}\left[ |\xi(0)|^2 (\log^+|\xi(0)|)^{ \chi} \right]}{M_n (\log M_n)^{ \chi}} \nonumber\\
& \le \frac{C}{\sqrt{n} (\log n)^{ \chi - \gamma/2}}.
\end{align}
Therefore, for $1 \le k \le \lfloor n T \rfloor$,
\begin{equation}
\frac{|Z^{(n)}_k - \widehat{Z}^{(n)}_k|}{\sqrt{n \log n}} = \frac{k \left| \mathbb{E}\left[ \xi(0) \mathbf{1}_{\{|\xi(0)| \le M_n\}} \right] \right|}{\sqrt{n
\log n}}
\le \frac{C \; T}{(\log n)^{ \chi + (1-\gamma)/2}}.
\end{equation}
This ends the proof, since $\chi + (1-\gamma)/2 = 3\chi/4 >0$.
\end{proof}

\subsection{Control of the truncated version}
\label{subsec:control_truncation}
From now on we will work with the truncated and recentered version $\widehat{Z}^{(n)}$ of the RWRS.
Let $\widehat{\mathcal{W}}^{(n)}$ be the analogue of $\mathcal{W}^{(n)}$ in \eqref{def:interpol} for $\widehat{Z}^{(n)}$, i.e.,
\begin{equation}\label{def:interpol2}
\widehat{\mathcal{W}}_t^{(n)} := \frac{\widehat{Z}^{(n)}_{\lfloor n t \rfloor} + \left(nt - \lfloor nt\rfloor \right) \left(\widehat{Z}^{(n)}_{\lfloor n t \rfloor + 1}-\widehat{Z}^{(n)}_{\lfloor n t \rfloor} \right)}{\sqrt{n \log n}}, \;\; t \ge 0.
\end{equation}

Fix $T >0$, $b \in (1,2]$ and $F:C([0,T], \R) \to \mathbb{C}$ bounded and Lipschitz.
By Propositions~\ref{prop:TvsNT}--\ref{prop:TvsTR}, weak convergence of either $W^{(n)}$ or $\widehat{W}^{(n)}$ implies the same convergence for the other, under both the quenched and annealed laws; 
therefore our work will be done once we show that
\begin{equation}\label{finalclaim}
\lim_{n \to \infty} \mathbb{E}\left[ F\left(\widehat{\mathcal{W}}^{(\lfloor b^n \rfloor)}\right) \middle| \xi \right]-\mathbb{E}\left[ F\left(\widehat{\mathcal{W}}^{(\lfloor b^n \rfloor)}\right) \right] = 0 \;\;\; \mathbb{P} \text{-a.s.}
\end{equation}

\begin{proof}[Proof of \eqref{finalclaim}]
Fix an arbitrary enumeration of $\Z^2 := \{x_1, x_2, \ldots\}$, define
\begin{equation}
\mathcal{G}_k := \sigma \left( \xi(x_i) \colon\, i \le k\right), \;\; k \in \N^*,
\end{equation}
and let
\begin{equation}
\Delta_k^{(n)}:= \mathbb{E}\left[ F\left(\widehat{\mathcal{W}}^{(n)}\right) \middle| \mathcal{G}_k \right] - \mathbb{E}\left[
F\left(\widehat{\mathcal{W}}^{(n)}\right) \middle| \mathcal{G}_{k-1} \right],
\end{equation}
where $\mathcal{G}_0$ is the trivial $\sigma$-algebra. The latter are increments of a bounded martingale. By the martingale convergence theorem,
\begin{equation}
\mathbb{E}\left[ F\left(\widehat{\mathcal{W}}^{(n)}\right) \middle|\; \xi \right]-\mathbb{E}\left[ F\left(\widehat{\mathcal{W}}^{(n)}\right) \right] = \sum_{k=1}^{\infty} \Delta_k^{(n)}.
\end{equation}

To control the $\Delta_k^{(n)}$, we introduce a coupling.
Let $\xi'$ be an independent copy of $\xi$, set
\begin{equation}
\widehat{\xi}_n^{(k)}(x):= \left\{ 
\begin{array}{ll}
\widehat{\xi'_n}(x) & \text{ if } x = x_k, \\
\widehat{\xi}_n(x) & \text{ otherwise,}
\end{array}\right.
\end{equation}
and let $\widehat{Z}^{(n,k)}$, $\widehat{\mathcal{W}}^{(n,k)}$ be the analogues of $\widehat{Z}^{(n)}$, $\widehat{\mathcal{W}}^{(n)}$,
but defined from $\widehat{\xi}_n^{(k)}$ and the same random walk $S$. Let $\mathbb{P}'$ denote
the joint law of $\xi'$, $\xi$ and $S$. Then
\begin{equation}\label{expr_Delta}
\Delta_k^{(n)} = \mathbb{E}' \left[ F\left(\widehat{\mathcal{W}}^{(n)} \right) - F \left(\widehat{\mathcal{W}}^{(n,k)} \right) \middle| \mathcal{G}_k\right] \;\;\; \mathbb{P}\text{-a.s.}
\end{equation}
Recalling \eqref{def:interpol2} and \eqref{defZ}, we see that
\begin{align}
\sup_{t \in [0,T]}|\widehat{\mathcal{W}}_t^{(n)} -\widehat{\mathcal{W}}_t^{(n,k)}|
& \le \sqrt{n \log n}^{-1} \sup_{1 \le m \le \lfloor nT \rfloor + 1} |\widehat{\xi}_n(x_k) - \widehat{\xi'_n}(x_k)|N_{m}(x_k)\nonumber\\
& = \sqrt{n \log n}^{-1} |\widehat{\xi}_n(x_k) - \widehat{\xi'_n}(x_k)|N_{\lfloor n T \rfloor+1}(x_k).
\end{align}
Therefore, by \eqref{expr_Delta}, the Lipschitz property of $F$ and Lemma~\ref{lemma:2dRWs}(i), we have
\begin{align}\label{bound_Delta}
|\Delta_k^{(n)}| & \le C \frac{M_n E \left[ N_{\lfloor nT \rfloor + 1}(x_k)\right]}{\sqrt{n \log n}} \nonumber\\
&\le C \frac{\log (nT +1)}{(\log n)^{\frac{\gamma+1}{2}}} \le \frac{C}{(\log n)^{ \chi/4}} \;\;\; \mathbb{P}\text{-a.s.},
\end{align}
and also
\begin{align}\label{bound_Delta2}
\mathbb{E}\left[|\Delta_k^{(n)}|^2 \middle | \mathcal{G}_{k-1}\right]
& \le C \frac{\mathbb{E}' \left[ \mathbb{E}' \left[|\widehat{\xi}_n(x_k) - \widehat{\xi'_n}(x_k)| \middle| \mathcal{G}_{k}\right]^2 \middle| \mathcal{G}_{k-1} \right] E \left[N_{\lfloor n T \rfloor+1}(x_k) \right]^2}{n \log n} \nonumber\\
& \le C \frac{\mathbb{E}' \left[|\widehat{\xi}_n(x_k) - \widehat{\xi'_n}(x_k)|^2 \right] E \left[N_{\lfloor n T \rfloor+1}(x_k) \right]^2}{n \log n} \nonumber\\
& \le C \frac{E \left[N_{\lfloor n T \rfloor+1}(x_k) \right]^2}{n \log n} \;\;\; \mathbb{P}\text{-a.s.},
\end{align}
where for the second line we used the Cauchy-Schwarz inequality. By \eqref{bound_Delta2} and Lemma~\ref{lemma:2dRWs}(ii), we have
\begin{equation}\label{bound_quadvar}
\sum_{k=1}^{\infty}\mathbb{E}\left[|\Delta_k^{(n)}|^2 \middle| \mathcal{G}_{k-1}\right]
\le C \frac{\lfloor nT \rfloor + 1}{ n \log n} \le \frac{C}{\log n} \;\;\; \mathbb{P}\text{-a.s.}
\end{equation}
Therefore, by Bernstein's inequality for martingales (see e.g.\ Theorem 1.2A in \cite{Pe99}),
for any $\epsilon >0$,
\begin{align}
\mathbb{P} \left( \left| \sum_{k=1}^{\infty} \Delta_k^{(n)} \right| > \epsilon \right)
& \le \exp \left\{ - C\frac{\epsilon^2}{(\log n)^{-1} + \epsilon(\log n)^{ -\chi/4} } \right\} \nonumber\\
& \le \exp\left\{ -C (\log n)^{ 1 \wedge \chi/4}\right\},
\end{align}
which is summable along $b^n$ for any $b > 1$; thus, by the Borel-Cantelli lemma, \eqref{finalclaim} holds.
\end{proof}


\noindent
{\bf Acknowledgments:}\\*
The authors are grateful to Mohamed El Machkouri and to Christophe Sabot for helpful and stimulating discussions.

\end{document}